\documentclass[11pt,reqno]{amsart}
\usepackage[no-math]{fontspec}
\usepackage{mathtools,amssymb,amsthm,mathrsfs}
\usepackage{microtype}
\usepackage{enumitem}
\usepackage[backend=biber,style=numeric,doi=true,sorting=nyt,eprint=true,giveninits=true,maxnames=99]{biblatex}
\DeclareFieldFormat[article]{title}{\mkbibemph{#1}}
\renewbibmacro*{in:}{\ifentrytype{article}{}{\printtext{\bibstring{in}\intitlepunct}}}
\usepackage[colorlinks=false,linkbordercolor={1 0 0},citebordercolor={0 1 0},urlbordercolor={0 1 1}]{hyperref}
\usepackage[nameinlink,noabbrev]{cleveref}
\newcommand{\PositiveIntegers}{\mathbb Z_{\ge 1}}
\newcommand{\NonnegativeIntegers}{\mathbb Z_{\ge 0}}
\newcommand{\Integers}{\mathbb Z}
\newcommand{\Rationals}{\mathbb Q}
\newcommand{\RealNumbers}{\mathbb R}
\DeclareMathOperator{\LinearSpan}{span}
\newcommand{\RealSpan}{\LinearSpan_{\RealNumbers}}
\DeclareMathOperator{\diam}{diam}
\DeclareMathOperator{\IsometryGroup}{Isom}
\DeclareMathOperator{\IdentityMap}{id}

\newcommand{\OpenBall}[3]{\mathrm{U}(#1,#2;#3)}
\newcommand{\ClosedBall}[3]{\mathrm{B}(#1,#2;#3)}

\newcommand{\CompatibleMetrics}{\operatorname{Met}}
\newcommand{\AlgebraicallyIndependentMetrics}{\mathrm{A}}
\newcommand{\RigidMetrics}{\mathrm{R}}
\newcommand{\ProperMetrics}{\CompatibleMetrics_{\mathrm{pr}}}
\newcommand{\TotallyBoundedMetrics}{\CompatibleMetrics_{\mathrm{tb}}}
\newcommand{\UniformDistance}{\mathcal{D}}

\newcommand{\GapExponent}{k}
\newcommand{\DigitVector}{\boldsymbol v}
\newcommand{\DigitSum}{\boldsymbol\xi}
\newcommand{\PartialSum}{\boldsymbol s}

\newcommand{\StandardBasisVector}{\boldsymbol e}

\newcommand{\IndicatorFunction}{\mathbf 1}
\newcommand{\CutPseudometric}{\kappa}
\newcommand{\ThresholdPseudometric}{\tau}
\newcommand{\CutVector}{\boldsymbol c}
\newcommand{\AuxiliaryMetric}{\rho}
\newcommand{\CoarsePseudometric}{q}
\newcommand{\PolynomialZeroMetrics}{\mathrm{F}}
\newcommand{\IsometryWitnessMetrics}{\mathrm{E}}

\newenvironment{acknowledgements}{\medskip\noindent\textit{Acknowledgements.}\ }{\par}
\newenvironment{useofai}{\medskip\noindent\textbf{Use of AI.}\ }{\par}
\numberwithin{equation}{section}

\theoremstyle{plain}
\newtheorem{theorem}{Theorem}[section]
\newtheorem{lemma}[theorem]{Lemma}
\newtheorem{proposition}[theorem]{Proposition}

\theoremstyle{definition}
\newtheorem{definition}[theorem]{Definition}

\newtheorem{question}[theorem]{Question}

\theoremstyle{remark}
\newtheorem{remark}[theorem]{Remark}

\crefname{theorem}{theorem}{theorems}
\Crefname{theorem}{Theorem}{Theorems}

\crefname{lemma}{lemma}{lemmas}
\Crefname{lemma}{Lemma}{Lemmas}

\crefname{proposition}{proposition}{propositions}
\Crefname{proposition}{Proposition}{Propositions}

\crefname{corollary}{corollary}{corollaries}
\Crefname{corollary}{Corollary}{Corollaries}

\crefname{claim}{claim}{claims}
\Crefname{claim}{Claim}{Claims}

\crefname{definition}{definition}{definitions}
\Crefname{definition}{Definition}{Definitions}

\crefname{example}{example}{examples}
\Crefname{example}{Example}{Examples}

\crefname{question}{question}{questions}
\Crefname{question}{Question}{Questions}

\crefname{remark}{remark}{remarks}
\Crefname{remark}{Remark}{Remarks}
\title{Algebraically independent distances and rigid metrics}
\author{Yoshito Ishiki}
\address{Department of Mathematical Sciences\\
Tokyo Metropolitan University\\
Minami-osawa, Hachioji, Tokyo 192-0397, Japan}
\email{ishiki-yoshito@tmu.ac.jp}
\date{September 15, 2026}
\begin{document}
\begin{abstract}
We study metrics whose distances on distinct two-point subsets are algebraically independent over the rationals. We prove that every compatible metric on a strongly zero-dimensional metrizable space of cardinality at most continuum can be uniformly approximated by compatible metrics with this property. If the space is completely metrizable, the approximating metrics can also be chosen complete. For every $\sigma$-compact metrizable space, the metrics with algebraically independent distances form a $G_\delta$ set in the uniform topology. We also study rigid metrics, whose only bijective self-isometry is the identity. On every locally compact Polish space, rigid proper metrics form a $G_\delta$ set among proper compatible metrics. Using a theorem of Niemiec, we obtain uniform density of rigid metrics on compact metrizable spaces with at least three points. Passing to compact completions then yields rigid approximations of every totally bounded compatible metric on any space with at least three points.
\end{abstract}
\maketitle
\section{Introduction}\label{sec:intro}
A metric on a set
$X$
is said to be strongly rigid if any two distinct two-point subsets of
$X$
have different distances.
The author proved that strongly rigid compatible metrics are uniformly dense on strongly zero-dimensional metrizable spaces of cardinality at most continuum \cite[Theorem~1.2]{Ishiki2024}.
The construction also ensures rational linear independence of any two distinct positive distances \cite[Theorem~1.1]{Ishiki2024}.
We consider algebraic independence for arbitrary finite collections of distinct two-point subsets.

For a metrizable space
$X$,
we denote by
$\CompatibleMetrics(X)$
the set of all metrics inducing its topology.
For
$d,e\in\CompatibleMetrics(X)$,
define
\[
\UniformDistance_X(d,e)=\sup_{x,y\in X}|d(x,y)-e(x,y)|.
\]
This function may take the value infinity.
The metric
$\min\{1,\UniformDistance_X\}$
induces the uniform topology considered throughout this paper.
We write
\[
[X]^2=\{\{x,y\}\mid x,y\in X,\ x\ne y\}.
\]

\begin{definition}\label{def:ai}
For a metrizable space
$X$,
we say that a metric
$d\in\CompatibleMetrics(X)$
has algebraically independent distances if,
for every positive integer
$r$,
every choice of distinct
$\{x_1,y_1\},\ldots,\{x_r,y_r\}\in[X]^2$,
and every nonzero polynomial
$P\in\Rationals[T_1,\ldots,T_r]$,
we have
\[
P(d(x_1,y_1),\ldots,d(x_r,y_r))\ne0.
\]
We denote by
$\AlgebraicallyIndependentMetrics(X)$
the set of all such metrics.
\end{definition}
For
$d\in\AlgebraicallyIndependentMetrics(X)$,
the values
$d(x,y)$,
indexed by
$\{x,y\}\in[X]^2$,
are transcendental.
Moreover,
for distinct
$\{x,y\},\{u,v\}\in[X]^2$,
we have
\[
d(x,y)-d(u,v)\ne0.
\]
A nonempty metrizable space is strongly zero-dimensional if its covering dimension is zero,
or equivalently,
if its large inductive dimension is zero.
We also regard the empty space as strongly zero-dimensional.
Let
$\mathfrak c$
denote the cardinality of the continuum.

The following theorem strengthens the two-distance conclusion to arbitrary finite algebraic relations.
\begin{theorem}\label{thm:density}
Let
$X$
be a strongly zero-dimensional metrizable space with
$|X|\le\mathfrak c$.
Then, for every
$d\in\CompatibleMetrics(X)$
and every
$\varepsilon>0$,
there exists
$e\in\AlgebraicallyIndependentMetrics(X)$
such that
$\UniformDistance_X(d,e)<\varepsilon$.
If
$X$
is completely metrizable, then
$e$
can be chosen complete.
\end{theorem}
The prescribed metric
$d$
need not be complete or bounded.
When
$|X|=\mathfrak c$,
the resulting family of algebraically independent distances also has cardinality
$\mathfrak c$.

We next prove that
$\AlgebraicallyIndependentMetrics(X)$
is a
$G_\delta$
subset of
$\CompatibleMetrics(X)$
for every
$\sigma$-compact metrizable space
$X$.
\begin{theorem}\label{thm:ai-gdelta}
Let
$X$
be a
$\sigma$-compact metrizable space.
Then
$\AlgebraicallyIndependentMetrics(X)$
is a
$G_\delta$
subset of
$\CompatibleMetrics(X)$.
If
$X$
is strongly zero-dimensional, then
$\AlgebraicallyIndependentMetrics(X)$
is dense in
$\CompatibleMetrics(X)$.
\end{theorem}
Indeed, every
$\sigma$-compact metrizable space has cardinality at most
$\mathfrak c$.

The author asked whether algebraically independent metrics are comeager on every strongly zero-dimensional metrizable space \cite[Question~5.3]{Ishiki2025Baire}.
Here we impose independence only on positive distances,
since every metric takes the value zero on the diagonal.
For strongly zero-dimensional
$\sigma$-compact
$X$,
Theorem~\ref{thm:ai-gdelta} shows that
$\AlgebraicallyIndependentMetrics(X)$
is a dense
$G_\delta$
subset of
$\CompatibleMetrics(X)$,
and hence is comeager.
This answers the positive-distance formulation of the question for
$\sigma$-compact spaces,
with the additional requirement that distinct two-point subsets have distinct distances.

The proof of Theorem~\ref{thm:density} has two parts.
We first construct a bounded compatible metric by summing clopen cut pseudometrics and pseudometrics defined by thresholds of a compatible ultrametric.
The cut pseudometrics supply enough linearly independent vectors on every finite collection of distinct two-point subsets to apply Lemma~\ref{lem:digits}.
The threshold pseudometrics ensure that the sum induces the topology of
$X$
and is complete when the chosen ultrametric is complete.
We then approximate the prescribed metric by a rational-valued pseudometric and add a small rational multiple of the independent metric.
This rational translation and scaling preserve algebraic independence.

Janos and Martin characterized nonempty zero-dimensional separable metrizable spaces by the existence of a totally bounded compatible star rigid metric \cite[Theorem~2]{JanosMartin1978}.
Here star rigidity means that, from each fixed point, distinct points have distinct distances.
In a different parameter space, Rouyer proved that a generic compact metric space is strongly rigid, using the term totally anisometric \cite[Theorem~2]{Rouyer2011}.
This theorem concerns isometry classes with the Gromov--Hausdorff topology.
Our approximation problem keeps the underlying topological space fixed.

A compatible metric is said to be rigid if its group of bijective self-isometries is trivial.
We write
\[
\RigidMetrics(X)=\{d\in\CompatibleMetrics(X)\mid\IsometryGroup(X,d)=\{\IdentityMap_X\}\}.
\]
For
$|X|\ge3$,
strong rigidity implies membership in
$\RigidMetrics(X)$.
Theorem~\ref{thm:density} therefore implies density of
$\RigidMetrics(X)$
when
$X$
is strongly zero-dimensional and
$3\le|X|\le\mathfrak c$.
Without a dimension assumption,
Niemiec's theorem on proper metrics \cite[Theorem~1.7]{Niemiec2014} implies density on compact metrizable spaces with at least three points.
We prove that
$\RigidMetrics(X)$
is
$G_\delta$
on compact spaces and, more generally, is relatively
$G_\delta$
within proper compatible metrics on locally compact Polish spaces.

\medskip\noindent\textbf{Organization.}
In Section~\ref{sec:digits},
we prove a criterion for algebraic independence that will be used in the metric construction.
Section~\ref{sec:construction} combines clopen cuts and partitions to construct the approximating metrics.
Section~\ref{sec:borel} treats polynomial relations on compact configurations.
Section~\ref{sec:rigidity} uses compactness of isometries and compact completions to study rigidity.
Section~\ref{sec:questions} formulates the remaining questions.

\medskip\noindent\textbf{Conventions and notation.}
All algebraic independence statements are over
$\Rationals$.
Sequences are indexed by
$\NonnegativeIntegers=\{0,1,2,\ldots\}$.
Vectors in finite-dimensional real spaces are written in bold.
For a subset
$U\subseteq X$,
the function
$\IndicatorFunction_U\colon X\to\{0,1\}$
is its indicator function.
For a metric space
$(X,d)$,
a point
$x\in X$,
and
$r>0$,
we denote the open and closed balls by
\[
\OpenBall{x}{r}{d}=\{y\in X\mid d(x,y)<r\},
\qquad
\ClosedBall{x}{r}{d}=\{y\in X\mid d(x,y)\le r\},
\]
respectively.
A proper metric is a metric for which every closed bounded ball is compact.
A Polish space is a separable completely metrizable space.
All isometry groups in this paper consist of bijective self-isometries.
For the empty space, the supremum defining
$\UniformDistance_X$
is taken to be zero.

\begin{useofai}
The author used OpenAI Codex for literature searches,
exploration of constructions and proofs,
English drafting,
and \LaTeX{} typesetting.
The author reviewed and revised the AI-assisted material and takes full responsibility for the content of this paper.
\end{useofai}

\begin{acknowledgements}
The author was supported by JSPS KAKENHI Grant Number JP24KJ0182.
\end{acknowledgements}
\section{A criterion for algebraic independence}\label{sec:digits}
To construct the distances in Theorem~\ref{thm:density},
we will sum pseudometrics with rapidly decreasing coefficients.
On a fixed finite collection of distinct two-point subsets,
their values form vectors with coordinates zero or one.
The next lemma reduces algebraic independence of the resulting sums to a spanning condition on the vectors that occur infinitely often.
The lemma also follows from Shiokawa's theorem \cite[Corollary~1]{Shiokawa1982},
as explained in Remark~\ref{rem:shiokawa},
but we include a proof to keep the paper self-contained.

\begin{lemma}\label{lem:digits}
Let
$b\ge 2$
and
$r\ge 1$
be integers.
Let
$\{\GapExponent_j\}_{j\in\NonnegativeIntegers}$
be a strictly increasing sequence of positive integers satisfying
\begin{equation}\label{eq:gap}
\lim_{j\to\infty}\frac{\GapExponent_{j+1}}{\GapExponent_j}=\infty.
\end{equation}
Let
$\{\DigitVector_j\}_{j\in\NonnegativeIntegers}$
be a sequence in
$\{0,1\}^r$,
and define
\[
\mathcal V=\{\DigitVector\in\{0,1\}^r\mid \DigitVector_j=\DigitVector\text{ for infinitely many }j\}.
\]
If
$\RealSpan\mathcal V=\RealNumbers^r$,
then the coordinates of the vector
\[
\DigitSum=\sum_{j=0}^{\infty}b^{-\GapExponent_j}\DigitVector_j
\]
are algebraically independent over
$\Rationals$.
\end{lemma}

\begin{proof}
Since the exponents are strictly increasing positive integers,
the series defining
$\DigitSum$
converges absolutely.
For the sake of contradiction,
suppose that a nonzero polynomial with rational coefficients vanishes at
$\DigitSum$.
Choose a nonzero polynomial
$P\in\Integers[T_1,\ldots,T_r]$
of the least total degree among all such polynomials.
Write its degree as
$d$.
A nonzero constant cannot vanish,
and hence
$d\ge 1$.

For every
$N\in\NonnegativeIntegers$,
define
\[
\PartialSum_N=\sum_{j=0}^{N}b^{-\GapExponent_j}\DigitVector_j.
\]
Each coordinate of
$\PartialSum_N$
belongs to
$b^{-\GapExponent_N}\Integers$.
Consequently,
for every
$N\in\NonnegativeIntegers$,
we have
\begin{equation}\label{eq:denominator}
P(\PartialSum_N)\in b^{-d\GapExponent_N}\Integers.
\end{equation}
All the vectors
$\PartialSum_N$
and
$\DigitSum$
belong to a fixed compact box in
$\RealNumbers^r$.
The gradient of
$P$
is bounded on that box.
Thus there exists a constant
$C>0$
such that,
for every
$N\in\NonnegativeIntegers$,
\begin{equation}\label{eq:tail}
\begin{split}
|P(\PartialSum_N)|
&=|P(\PartialSum_N)-P(\DigitSum)|
\le C\|\PartialSum_N-\DigitSum\|_{\infty}\\
&\le C\sum_{j>N}b^{-\GapExponent_j}
\le \frac{Cb}{b-1}b^{-\GapExponent_{N+1}}.
\end{split}
\end{equation}
The last inequality uses strict increase of the integer exponents.
By
\eqref{eq:gap},
we have
$\GapExponent_{N+1}-d\GapExponent_N\to\infty$.
Combining
\eqref{eq:denominator}
and
\eqref{eq:tail},
we obtain
$N_*\in\NonnegativeIntegers$
such that
\begin{equation}\label{eq:eventual-zero}
P(\PartialSum_N)=0\qquad(N\ge N_*).
\end{equation}
Indeed,
the absolute value of
$P(\PartialSum_N)$
is eventually smaller than the least possible positive value
$b^{-d\GapExponent_N}$.

Fix
$\DigitVector\in\mathcal V$.
Choose an increasing sequence
$\{N_\ell\}_{\ell\in\NonnegativeIntegers}$
such that
$N_\ell>N_*$
and
$\DigitVector_{N_\ell}=\DigitVector$
for every
$\ell\in\NonnegativeIntegers$.
For every
$\ell\in\NonnegativeIntegers$,
put
$\varepsilon_\ell=b^{-\GapExponent_{N_\ell}}$.
Applying the chain rule and the fundamental theorem of calculus to the function
$t\mapsto P(\PartialSum_{N_\ell-1}+t\varepsilon_\ell\DigitVector)$
on
$[0,1]$,
and using
$\PartialSum_{N_\ell}=\PartialSum_{N_\ell-1}+\varepsilon_\ell \DigitVector$
and
\eqref{eq:eventual-zero},
we obtain
\begin{equation}\label{eq:derivative}
0=\frac{P(\PartialSum_{N_\ell})-P(\PartialSum_{N_\ell-1})}{\varepsilon_\ell}
=\int_0^1\nabla P(\PartialSum_{N_\ell-1}+t\varepsilon_\ell \DigitVector)\cdot \DigitVector\,dt.
\end{equation}
The vectors in the integrand converge to
$\DigitSum$
uniformly for
$t\in[0,1]$,
since
\[
\sup_{t\in[0,1]}\|\PartialSum_{N_\ell-1}+t\varepsilon_\ell \DigitVector-\DigitSum\|_\infty
\le \|\PartialSum_{N_\ell-1}-\DigitSum\|_\infty+\varepsilon_\ell\longrightarrow 0.
\]
Continuity of the polynomial gradient and
\eqref{eq:derivative}
therefore imply
$\nabla P(\DigitSum)\cdot \DigitVector=0$.
This holds for every
$\DigitVector\in\mathcal V$.
Since
$\mathcal V$
spans
$\RealNumbers^r$,
we have
\[
\nabla P(\DigitSum)=0.
\]
Since
$P$
is nonconstant,
there exists
$i\in\{1,\ldots,r\}$
such that the polynomial
$Q=\partial P/\partial T_i$
satisfies
\[
Q\in\Integers[T_1,\ldots,T_r]\setminus\{0\},
\qquad \deg Q<d,
\qquad Q(\DigitSum)=0.
\]
This contradicts the minimality of
$d$.
\end{proof}

\begin{remark}\label{rem:shiokawa}
To deduce Lemma~\ref{lem:digits} from \cite[Corollary~1, p.~440]{Shiokawa1982},
choose a basis from the coefficient vectors that occur infinitely often.
An invertible rational change of coordinates sends these basis vectors to integer multiples of the standard basis vectors while keeping all coefficients integral.
Thus each transformed coordinate series has infinitely many exponents that occur in no other coordinate series.
Together with the gap condition \eqref{eq:gap},
this allows Shiokawa's criterion to be applied at
$z=1/b$,
yielding the conclusion of the lemma.
See also \cite[Example~1, p.~441]{Shiokawa1982} for an application to series with factorial exponents.
\end{remark}
\section{Construction of metrics with algebraically independent distances}\label{sec:construction}
We apply Lemma~\ref{lem:digits} to construct a bounded metric in
$\AlgebraicallyIndependentMetrics(X)$.
A family
$\mathcal B$
of subsets of a set
$X$
is said to be a Boolean algebra if it contains
$\emptyset$
and
$X$
and is closed under finite unions,
finite intersections,
and complements relative to
$X$.
We first construct a countable Boolean algebra of clopen sets that realizes every subset of every finite set of points.

A metric
$\sigma$
on
$X$
is said to be an ultrametric if,
for all
$x,y,z\in X$,
we have
\[
\sigma(x,z)\le\max\{\sigma(x,y),\sigma(y,z)\}.
\]
We use its balls to construct the required partitions.

\begin{lemma}\label{lem:cuts}
Let
$X$
be a strongly zero-dimensional metrizable space with
$2\le|X|\le\mathfrak c$.
Then there exists a countable Boolean algebra
$\mathcal B$
of clopen subsets of
$X$
such that,
for every finite set
$V\subseteq X$
and every
$H\subseteq V$,
there exists
$B\in\mathcal B$
with
$B\cap V=H$.
\end{lemma}
\begin{proof}
By \cite[Theorem~II, p.~951]{deGroot1956},
choose a compatible ultrametric
$\sigma$
on
$X$.
For every
$n\in\NonnegativeIntegers$,
define
\[
\mathcal P_n=\{\OpenBall{y}{2^{-n}}{\sigma}\mid y\in X\}.
\]
The ultrametric inequality implies that two balls of the same radius are either equal or disjoint.
Thus
$\mathcal P_n$
is a partition into open sets.
The complement of each member is a union of other members,
so every member is clopen.
For
$n\in\NonnegativeIntegers$
and
$x\in X$,
write
$\mathcal P_n(x)=\OpenBall{x}{2^{-n}}{\sigma}$.
Since
$2^{-n}\to0$,
these partitions separate points.

We next construct the countable algebra.
For each
$n\in\NonnegativeIntegers$,
the members of
$\mathcal P_n$
are nonempty and mutually disjoint,
so
$|\mathcal P_n|\le|X|\le\mathfrak c$.
Choose an injection
$\theta_n\colon\mathcal P_n\to\{0,1\}^{\NonnegativeIntegers}$.
For
$n\in\NonnegativeIntegers$
and
$k\in\NonnegativeIntegers$,
define
\[
U_{n,k}=\{x\in X\mid\theta_n(\mathcal P_n(x))_k=1\}.
\]
Both
$U_{n,k}$
and its complement are unions of members of
$\mathcal P_n$,
so
$U_{n,k}$
is clopen.
Since the partitions separate points and each
$\theta_n$
is injective,
the family
$\{U_{n,k}\}_{n\in\NonnegativeIntegers,\,k\in\NonnegativeIntegers}$
separates points.
Let
$\mathcal B$
be the Boolean algebra generated by this family.
It is countable.
Fix a finite set
$V\subseteq X$.
For distinct
$v,w\in V$,
choose a generating set or its complement
$B_{v,w}\in\mathcal B$
such that
$v\in B_{v,w}$
and
$w\notin B_{v,w}$.
For every
$v\in V$,
put
\[
B_v=\bigcap_{w\in V\setminus\{v\}}B_{v,w}.
\]
Then
$B_v\cap V=\{v\}$.
For each
$H\subseteq V$,
the set
$B=\bigcup_{v\in H}B_v$
belongs to
$\mathcal B$
and satisfies
$B\cap V=H$.
\end{proof}
We use the countable algebra to obtain algebraic independence.
Pseudometrics defined by thresholds of a compatible ultrametric will ensure compatibility and completeness.

\begin{proposition}\label{prop:bounded-ai}
Let
$X$
be a strongly zero-dimensional metrizable space with
$|X|\le\mathfrak c$.
Then there exists
$\AuxiliaryMetric\in\AlgebraicallyIndependentMetrics(X)$
with
$\AuxiliaryMetric(x,y)\le1/2$
for all
$x,y\in X$.
If
$X$
is completely metrizable, then
$\AuxiliaryMetric$
can be chosen complete.
\end{proposition}
\begin{proof}
For
$|X|\le1$,
the assertion holds for the unique metric.
Assume that
$|X|\ge2$.
Choose
$\mathcal B$
by Lemma~\ref{lem:cuts} and a compatible ultrametric
$\sigma$
by \cite[Theorem~II]{deGroot1956}.
If
$X$
is completely metrizable,
choose
$\sigma$
complete by \cite[Proposition~2.18]{Ishiki2021Ultrametrics}.
For
$U\in\mathcal B$
and
$x,y\in X$,
define
\[
\CutPseudometric_U(x,y)=|\IndicatorFunction_U(x)-\IndicatorFunction_U(y)|.
\]
For
$n\in\NonnegativeIntegers$
and
$x,y\in X$,
define
\[
\ThresholdPseudometric_n(x,y)=\begin{cases}0&\sigma(x,y)<2^{-n},\\1&\sigma(x,y)\ge2^{-n}.\end{cases}
\]
The cut function
$\CutPseudometric_U$
is a continuous pseudometric.
For every
$n\in\NonnegativeIntegers$
and every
$x,y,z\in X$,
the ultrametric inequality implies
\[
\ThresholdPseudometric_n(x,z)\le\max\{\ThresholdPseudometric_n(x,y),\ThresholdPseudometric_n(y,z)\}.
\]
Thus
$\ThresholdPseudometric_n$
is a pseudometric.
For every
$n\in\NonnegativeIntegers$,
every
$x,y\in X$,
and every
$(u,v)\in\OpenBall{x}{2^{-n}}{\sigma}\times\OpenBall{y}{2^{-n}}{\sigma}$,
the ultrametric inequality also implies
\[
\ThresholdPseudometric_n(u,v)=\ThresholdPseudometric_n(x,y).
\]
Hence
$\ThresholdPseudometric_n$
is locally constant and therefore continuous.

Choose a sequence
$\{\delta_j\}_{j\in\NonnegativeIntegers}$
from the countable collection
\[
\{\CutPseudometric_U\mid U\in\mathcal B\}
\cup\{\ThresholdPseudometric_n\mid n\in\NonnegativeIntegers\}
\]
in which every member of this collection occurs infinitely often.
For
$j\in\NonnegativeIntegers$,
put
$w_j=2^{-(j+2)!}$.
For
$x,y\in X$,
define
\begin{equation}\label{eq:series-metric}
\AuxiliaryMetric(x,y)=\sum_{j=0}^{\infty}w_j\delta_j(x,y).
\end{equation}
For every
$N\in\NonnegativeIntegers$,
we have
\[
\sup_{x,y\in X}\left|
\AuxiliaryMetric(x,y)-\sum_{j=0}^Nw_j\delta_j(x,y)\right|
\le\sum_{j>N}w_j\longrightarrow0\quad(N\to\infty).
\]
Thus the convergence is uniform.
Also,
\[
\sum_{j=0}^{\infty}w_j\le\sum_{j=0}^{\infty}2^{-(j+2)}=\frac12.
\]
Thus
$\AuxiliaryMetric$
is a continuous pseudometric.
For each
$n\in\NonnegativeIntegers$,
choose
$j(n)$
such that
$\delta_{j(n)}=\ThresholdPseudometric_n$.
For every
$n\in\NonnegativeIntegers$
and every
$x,y\in X$,
we have
\begin{equation}\label{eq:threshold-lower}
\AuxiliaryMetric(x,y)\ge w_{j(n)}\ThresholdPseudometric_n(x,y).
\end{equation}
For distinct
$x,y\in X$,
choose
$n\in\NonnegativeIntegers$
with
$2^{-n}\le\sigma(x,y)$.
Then \eqref{eq:threshold-lower} implies
$\AuxiliaryMetric(x,y)\ge w_{j(n)}>0$,
so
$\AuxiliaryMetric$
is a metric.
For every
$n\in\NonnegativeIntegers$
and every
$x\in X$,
\eqref{eq:threshold-lower} implies
\begin{equation}\label{eq:ultrametric-ball-inclusion}
\OpenBall{x}{w_{j(n)}}{\AuxiliaryMetric}\subseteq\OpenBall{x}{2^{-n}}{\sigma}.
\end{equation}
Since
$\sigma$
is compatible and
$2^{-n}\to0$,
this inclusion and continuity of
$\AuxiliaryMetric$
prove compatibility.

We next prove algebraic independence.
Fix a positive integer
$r$
and distinct two-point subsets
$\{x_1,y_1\},\ldots,\{x_r,y_r\}\in[X]^2$,
and put
$V=\bigcup_{i=1}^r\{x_i,y_i\}$.
For
$H\subseteq V$,
define
$\CutVector_H\in\{0,1\}^r$
by
\[
\CutVector_H=(|\IndicatorFunction_H(x_i)-\IndicatorFunction_H(y_i)|)_{i=1}^r.
\]
Fix
$H\subseteq V$.
By Lemma~\ref{lem:cuts},
choose
$B\in\mathcal B$
with
$B\cap V=H$.
Since all the points
$x_i,y_i$
belong to
$V$,
for every
$i\in\{1,\ldots,r\}$
we have
\[
\CutPseudometric_B(x_i,y_i)
=|\IndicatorFunction_B(x_i)-\IndicatorFunction_B(y_i)|
=|\IndicatorFunction_H(x_i)-\IndicatorFunction_H(y_i)|.
\]
Thus the values of the single pseudometric
$\CutPseudometric_B$
on the selected two-point subsets form the vector
$\CutVector_H$.
By the choice of the sequence
$\{\delta_j\}_{j\in\NonnegativeIntegers}$,
we have
$\delta_j=\CutPseudometric_B$
for infinitely many
$j$.
At each such index,
\[
(\delta_j(x_i,y_i))_{i=1}^r=\CutVector_H.
\]
Since
$H$
was arbitrary,
every cut vector occurs infinitely often.
We now show that these cut vectors span
$\RealNumbers^r$.
For each
$i\in\{1,\ldots,r\}$,
let
$\StandardBasisVector_i$
be the
$i$-th standard basis vector of
$\RealNumbers^r$.
Then
\begin{equation}\label{eq:cut-span}
\CutVector_{\{x_i\}}+\CutVector_{\{y_i\}}-\CutVector_{\{x_i,y_i\}}=2\StandardBasisVector_i.
\end{equation}
Indeed,
fix
$i\in\{1,\ldots,r\}$.
For each
$h\in\{1,\ldots,r\}$,
the
$h$-th coordinate of the left side is
\[
\begin{aligned}
L_h
&=|\IndicatorFunction_{\{x_i\}}(x_h)-\IndicatorFunction_{\{x_i\}}(y_h)|
+|\IndicatorFunction_{\{y_i\}}(x_h)-\IndicatorFunction_{\{y_i\}}(y_h)|\\
&\quad-|\IndicatorFunction_{\{x_i,y_i\}}(x_h)-\IndicatorFunction_{\{x_i,y_i\}}(y_h)|.
\end{aligned}
\]
Substituting
$h=i$
and using
$x_i\ne y_i$,
we obtain
\[
L_i=|1-0|+|0-1|-|1-1|=2.
\]
Now let
$h\ne i$.
The distinct two-point subsets
$\{x_h,y_h\}$
and
$\{x_i,y_i\}$
have at most one common point.
If they have one common point,
we may interchange the endpoints within either two-point subset and assume that
$x_h=x_i$
and
$y_h\notin\{x_i,y_i\}$.
Substitution then yields
\[
L_h=|1-0|+|0-0|-|1-0|=0.
\]
If they have no common point,
all six indicator values are zero,
so
\[
L_h=|0-0|+|0-0|-|0-0|=0.
\]
Thus
$L_i=2$
and
$L_h=0$
for all
$h\ne i$,
which proves \eqref{eq:cut-span}.
Hence the recurrent vectors span
$\RealNumbers^r$.
Lemma~\ref{lem:digits}, with
$b=2$
and
$\GapExponent_j=(j+2)!$,
proves that
$\AuxiliaryMetric(x_1,y_1),\ldots,\AuxiliaryMetric(x_r,y_r)$
are algebraically independent.

Assume now that
$X$
is completely metrizable,
and let
$\{x_j\}_{j\in\NonnegativeIntegers}$
be a Cauchy sequence in
$(X,\AuxiliaryMetric)$.
For every
$n\in\NonnegativeIntegers$,
there exists
$J\in\NonnegativeIntegers$
such that,
for all
$j,\ell\ge J$,
we have
\[
\AuxiliaryMetric(x_j,x_\ell)<w_{j(n)}.
\]
By \eqref{eq:ultrametric-ball-inclusion},
we obtain
\[
\sigma(x_j,x_\ell)<2^{-n}\qquad(j,\ell\ge J).
\]
Thus the sequence is Cauchy in
$(X,\sigma)$.
Completeness of
$\sigma$
yields a limit
$x\in X$.
Since
$\sigma$
and
$\AuxiliaryMetric$
induce the same topology,
we have
$\AuxiliaryMetric(x_j,x)\to0$.
Hence
$\AuxiliaryMetric$
is complete.
\end{proof}

We next prove Theorem~\ref{thm:density} using the metric constructed in Proposition~\ref{prop:bounded-ai}.

\begin{remark}\label{rem:gluing}
The construction in \cite[Proposition~5.1]{Ishiki2024} cannot be applied directly here,
since its defining formula imposes algebraic relations among the distances.
\end{remark}

The rational approximation below uses the ceiling construction of \cite[Proposition~2.5]{Ishiki2023Ranges} on one representative from each member of a clopen partition.
We then add a rational multiple of the metric from Proposition~\ref{prop:bounded-ai}.

\begin{proof}[Proof of Theorem~\ref{thm:density}]
The case
$|X|\le1$
is immediate.
Fix
$d\in\CompatibleMetrics(X)$
and
$\varepsilon>0$,
and choose
$\eta\in\Rationals$
with
$0<\eta<\varepsilon/4$.
Choose
$\AuxiliaryMetric$
by Proposition~\ref{prop:bounded-ai},
complete if
$X$
is completely metrizable.
By \cite[Proposition~1.2 and Corollary~1.4]{Ellis1970},
choose a clopen partition
$\mathcal C$
refining the open cover
$\{\OpenBall{x}{\eta/3}{d}\}_{x\in X}$.
Then,
for every
$C\in\mathcal C$,
we have
\[
\diam_d C\le\frac{2\eta}{3}<\eta.
\]
For each
$C\in\mathcal C$,
choose
$p_C\in C$.
Define the map
$\CoarsePseudometric\colon X^2\to[0,\infty)$
by setting,
for
$C,D\in\mathcal C$,
$x\in C$,
and
$y\in D$,
\[
\CoarsePseudometric(x,y)=\eta\left\lceil\frac{d(p_C,p_D)}\eta\right\rceil.
\]
This map is locally constant and satisfies
\[
\CoarsePseudometric(X^2)\subseteq\eta\NonnegativeIntegers\subseteq\Rationals_{\ge0}.
\]
The triangle inequality for
$\CoarsePseudometric$
follows from that for
$d$
and the subadditivity of the ceiling function.
Since
$\CoarsePseudometric$
is symmetric and vanishes on the diagonal,
it is a continuous pseudometric.
For every
$C,D\in\mathcal C$,
$x\in C$,
and
$y\in D$,
we obtain
\begin{equation}\label{eq:coarse-error}
\begin{split}
|\CoarsePseudometric(x,y)-d(x,y)|
&\le\left|\eta\left\lceil\frac{d(p_C,p_D)}\eta\right\rceil-d(p_C,p_D)\right|
+d(x,p_C)+d(y,p_D)\\
&<3\eta.
\end{split}
\end{equation}
Define
$e=\CoarsePseudometric+\eta \AuxiliaryMetric$.
This is a continuous metric, and
$e\ge\eta \AuxiliaryMetric$
proves that it is compatible with the topology of
$X$.
By \eqref{eq:coarse-error} and the bound on
$\AuxiliaryMetric$,
we obtain
\[
\begin{aligned}
\UniformDistance_X(d,e)
&=\sup_{x,y\in X}|d(x,y)-\CoarsePseudometric(x,y)-\eta\AuxiliaryMetric(x,y)|\\
&\le\sup_{x,y\in X}|d(x,y)-\CoarsePseudometric(x,y)|
+\eta\sup_{x,y\in X}\AuxiliaryMetric(x,y)\\
&\le3\eta+\frac\eta2<\varepsilon.
\end{aligned}
\]
Fix
$r\in\PositiveIntegers$
and distinct
$\{x_1,y_1\},\ldots,\{x_r,y_r\}\in[X]^2$.
Then
the substitution
\[
T_i\longmapsto \CoarsePseudometric(x_i,y_i)+\eta T_i\qquad(1\le i\le r)
\]
is an automorphism of
$\Rationals[T_1,\ldots,T_r]$.
Since the coordinates
$\AuxiliaryMetric(x_i,y_i)$
are algebraically independent, so are the coordinates
$e(x_i,y_i)$.
Thus
$e\in\AlgebraicallyIndependentMetrics(X)$.
Assume now that
$X$
is completely metrizable.
Then the chosen
$\AuxiliaryMetric$
is complete.
Every Cauchy sequence in
$(X,e)$
is Cauchy in
$(X,\AuxiliaryMetric)$
because
$e\ge\eta \AuxiliaryMetric$.
It therefore converges in
$(X,\AuxiliaryMetric)$.
Since both metrics induce the topology of
$X$,
the sequence also converges in
$(X,e)$,
proving completeness of
$e$.
\end{proof}

\begin{remark}\label{rem:magnitude}
Algebraic independence of the distances also permits reconstruction of each finite subspace from its magnitude function.
For a nonempty finite metric space
$(A,e)$,
define its similarity matrix and magnitude function, for sufficiently large
$t>0$,
by
\[
Z_A(t)=(\exp(-t e(x,y)))_{x,y\in A},
\qquad
M_A(t)=\sum_{x,y\in A}(Z_A(t)^{-1})_{xy}.
\]
The matrix tends to the identity as
$t\to\infty$,
so it is invertible for all sufficiently large
$t$.
O'Hara proves that rational independence of the edge lengths allows recovery of the finite metric space from this function \cite[Theorem~2.4(2)]{OHara2025}.
For
$e\in\AlgebraicallyIndependentMetrics(X)$,
every nonempty finite subspace
$A\subseteq X$
satisfies this hypothesis, since algebraic independence implies rational linear independence.
Thus its magnitude function determines its isometry type.
\end{remark}
\section{Algebraic independence as a \texorpdfstring{$G_\delta$}{G-delta} property}\label{sec:borel}
Theorem~\ref{thm:density} constructs algebraically independent metrics.
We now describe the set of all such metrics when
$X$
is
$\sigma$-compact.
To obtain closed sets of metrics satisfying a polynomial relation,
we restrict its endpoint tuples to compact sets and keep distinct two-point subsets from coalescing.
For equality of two distances,
related compactness arguments appear in \cite[Theorem~2]{Rouyer2011} and \cite[Proposition~5.3]{Ishiki2024}.

\begin{proof}[Proof of Theorem~\ref{thm:ai-gdelta}]
Fix a compatible metric
$\rho$
and an increasing sequence of compact sets
$\{K_n\}_{n\in\NonnegativeIntegers}$
with
$X=\bigcup_{n\in\NonnegativeIntegers}K_n$.
For points
$x,y,u,v\in X$,
define
\[
\Delta_\rho((x,y),(u,v))=\min\{\rho(x,u)+\rho(y,v),\rho(x,v)+\rho(y,u)\}.
\]
Fix
$r\in\PositiveIntegers$
and
$n,m\in\NonnegativeIntegers$,
and let
$\mathcal C_{r,n,m}$
be the set of tuples
$((x_i,y_i))_{i=1}^r\in K_n^{2r}$
satisfying
\[
\rho(x_i,y_i)\ge2^{-m}\quad(1\le i\le r),
\]
\[
\Delta_\rho((x_i,y_i),(x_h,y_h))\ge2^{-m}\quad(1\le i<h\le r).
\]
This is a compact set.
Every tuple representing distinct two-point subsets belongs to some
$\mathcal C_{r,n,m}$.

For a nonzero
$P\in\Rationals[T_1,\ldots,T_r]$,
let
$\PolynomialZeroMetrics_{r,P,n,m}$
consist of the metrics
$d\in\CompatibleMetrics(X)$
for which some tuple in
$\mathcal C_{r,n,m}$
satisfies
\[
P(d(x_1,y_1),\ldots,d(x_r,y_r))=0.
\]
We prove that
$\PolynomialZeroMetrics_{r,P,n,m}$
is closed.
Let
$d_\ell\in \PolynomialZeroMetrics_{r,P,n,m}$
and
$d\in\CompatibleMetrics(X)$
satisfy
\[
\UniformDistance_X(d_\ell,d)\longrightarrow0\quad(\ell\to\infty).
\]
For each
$\ell\in\NonnegativeIntegers$,
choose a witnessing tuple
$((x_{i,\ell},y_{i,\ell}))_{i=1}^r\in\mathcal C_{r,n,m}$.
By compactness,
pass to a subsequence converging to a tuple
$((x_i,y_i))_{i=1}^r\in\mathcal C_{r,n,m}$.
For each
$i\in\{1,\ldots,r\}$,
continuity of
$d$
implies
\[
|d_\ell(x_{i,\ell},y_{i,\ell})-d(x_i,y_i)|
\le \UniformDistance_X(d_\ell,d)+|d(x_{i,\ell},y_{i,\ell})-d(x_i,y_i)|\longrightarrow0.
\]
Continuity of
$P$
implies
\[
P(d(x_1,y_1),\ldots,d(x_r,y_r))
=\lim_{\ell\to\infty}P(d_\ell(x_{1,\ell},y_{1,\ell}),\ldots,d_\ell(x_{r,\ell},y_{r,\ell}))=0.
\]
Hence
$d\in\PolynomialZeroMetrics_{r,P,n,m}$.
Thus
\[
\AlgebraicallyIndependentMetrics(X)=\bigcap_{r\ge1}\ \bigcap_{0\ne P\in\Rationals[T_1,\ldots,T_r]}\ \bigcap_{n,m\in\NonnegativeIntegers}
\bigl(\CompatibleMetrics(X)\setminus \PolynomialZeroMetrics_{r,P,n,m}\bigr)
\]
is a countable intersection of open sets.
If
$X$
is also strongly zero-dimensional, then
$|X|\le\mathfrak c$
and Theorem~\ref{thm:density} proves density.
\end{proof}
\section{Metrics with trivial isometry group}\label{sec:rigidity}
We now study rigidity without a dimension assumption.
On a locally compact Polish space,
we prove that the nonrigid proper metrics form an
$F_\sigma$
subset of the space of proper compatible metrics.
The proof takes limits of isometries and their inverses,
using properness to place the images of each point in a compact set.

\begin{remark}\label{rem:zero-dimensional-rigidity}
For a strongly zero-dimensional metrizable space
$X$
with
$3\le|X|\le\mathfrak c$,
Theorem~\ref{thm:density} produces rigid metrics.
Indeed,
strong rigidity forces an isometry to preserve each two-point subset,
and the intersection of two such subsets determines their common point.
\end{remark}

We denote by
$\ProperMetrics(X)$
the proper compatible metrics and by
$\TotallyBoundedMetrics(X)$
the totally bounded compatible metrics.
Both sets carry the subspace topology of
$(\CompatibleMetrics(X),\UniformDistance_X)$.

\begin{proposition}\label{prop:proper-gdelta}
Let
$X$
be a locally compact Polish space.
Then
$\RigidMetrics(X)\cap\ProperMetrics(X)$
is a
$G_\delta$
subset of
$\ProperMetrics(X)$.
\end{proposition}
\begin{proof}
The empty space is immediate.
Fix a point
$o\in X$,
a dense sequence
$\{a_j\}_{j\in\NonnegativeIntegers}$,
and an increasing sequence of compact sets
$\{K_n\}_{n\in\NonnegativeIntegers}$
covering
$X$.
For
$j,m,n\in\NonnegativeIntegers$,
let
$\IsometryWitnessMetrics_{j,m,n}$
consist of the metrics
$d\in\ProperMetrics(X)$
admitting
$f\in\IsometryGroup(X,d)$
such that
\begin{enumerate}[label=\textup{(R\arabic*)},ref=\textup{(R\arabic*)}]
\item\label{cond:rigid-basepoint}
$f(o),f^{-1}(o)\in K_n$.
\item\label{cond:rigid-displacement}
$d(a_j,f(a_j))\ge2^{-m}$.
\end{enumerate}
Every continuous nonidentity self-map moves a member of the dense sequence.
Consequently,
\begin{equation}\label{eq:rigid-complement}
\ProperMetrics(X)\setminus\RigidMetrics(X)=\bigcup_{j,m,n\in\NonnegativeIntegers}\IsometryWitnessMetrics_{j,m,n}.
\end{equation}

Fix
$j,m,n\in\NonnegativeIntegers$.
We prove that
$\IsometryWitnessMetrics_{j,m,n}$
is relatively closed.
Let
$d_k\in \IsometryWitnessMetrics_{j,m,n}$
converge uniformly to
$d\in\ProperMetrics(X)$,
choose witnessing isometries
$f_k$,
and put
$g_k=f_k^{-1}$
and
$\eta_k=\UniformDistance_X(d_k,d)$.
After discarding finitely many terms and relabeling,
assume that
$\eta_k\le1$
for every
$k\in\NonnegativeIntegers$.
We first bound the images of each point under the maps and their inverses.
For every
$k\in\NonnegativeIntegers$,
every
$h\in\{f_k,g_k\}$,
and every
$x,y\in X$,
we have
\begin{equation}\label{eq:approx-isometry}
\begin{aligned}
&|d(h(x),h(y))-d(x,y)|\\
&\quad\le |d(h(x),h(y))-d_k(h(x),h(y))|
+|d_k(h(x),h(y))-d(x,y)|\\
&\quad=|d(h(x),h(y))-d_k(h(x),h(y))|
+|d_k(x,y)-d(x,y)|\\
&\quad\le2\eta_k.
\end{aligned}
\end{equation}
For every
$k\in\NonnegativeIntegers$,
every
$h\in\{f_k,g_k\}$,
and every
$x\in X$,
condition~\ref{cond:rigid-basepoint} and \eqref{eq:approx-isometry} imply
\begin{equation}\label{eq:proper-bound}
\begin{aligned}
d(o,h(x))
&\le d(o,h(o))+d(h(o),h(x))\\
&\le\sup_{z\in K_n}d(o,z)+d(o,x)+2\eta_k.
\end{aligned}
\end{equation}
For each fixed
$x\in X$,
put
$R_x=\sup_{z\in K_n}d(o,z)+d(o,x)+2$.
By \eqref{eq:proper-bound},
for every
$k\in\NonnegativeIntegers$
we have
\[
f_k(x),g_k(x)\in\ClosedBall{o}{R_x}{d}.
\]
This set is compact because
$d$
is proper.

We next construct the limiting isometric embeddings.
Put
$A=\{o\}\cup\{a_j\mid j\in\NonnegativeIntegers\}$.
Since
$A$
is countable,
a diagonal argument applied to the pairs
$(f_k(a),g_k(a))$
yields a single subsequence along which both
$f_k(a)$
and
$g_k(a)$
converge for every
$a\in A$.
Define the maps
$f_A,g_A\colon A\to X$
by the following formulas for
$a\in A$.
\[
f_A(a)=\lim_{k\to\infty}f_k(a),
\qquad g_A(a)=\lim_{k\to\infty}g_k(a).
\]
Equation~\eqref{eq:approx-isometry} shows that
$f_A,g_A\colon A\to X$
preserve
$d$.
Since
$A$
is dense and
$d$
is complete,
they extend uniquely to isometric embeddings
$f,g\colon X\to X$.
We prove pointwise convergence on all of
$X$.
For every
$x\in X$,
$a\in A$,
and
$k\in\NonnegativeIntegers$,
the triangle inequality,
\eqref{eq:approx-isometry},
and the fact that
$f$
preserves
$d$
imply
\begin{equation}\label{eq:pointwise-estimates}
\begin{aligned}
d(f_k(x),f(x))
&\le d(f_k(x),f_k(a))+d(f_k(a),f(a))+d(f(a),f(x))\\
&\le d(x,a)+2\eta_k+d(f_k(a),f(a))+d(a,x)\\
&=2d(x,a)+2\eta_k+d(f_k(a),f(a)).
\end{aligned}
\end{equation}
Similarly,
for every
$x\in X$,
$a\in A$,
and
$k\in\NonnegativeIntegers$,
we obtain
\begin{equation}\label{eq:inverse-pointwise-estimate}
d(g_k(x),g(x))\le2d(x,a)+2\eta_k+d(g_k(a),g(a)).
\end{equation}
Fix
$x\in X$.
Since both sequences converge on
$A$,
for every
$a\in A$
we obtain
\[
\limsup_{k\to\infty}\max\{d(f_k(x),f(x)),d(g_k(x),g(x))\}\le2d(x,a).
\]
The density of
$A$
therefore implies
\[
f_k(x)\longrightarrow f(x),\qquad g_k(x)\longrightarrow g(x).
\]
It remains to prove that
$f$
and
$g$
are inverse maps.
For each
$x\in X$,
applying \eqref{eq:approx-isometry} to
$g_k$
yields
\[
d(g_k(f_k(x)),g_k(f(x)))\le d(f_k(x),f(x))+2\eta_k\longrightarrow0.
\]
Since
$g_k(f(x))\to g(f(x))$
and
$g_k(f_k(x))=x$,
we obtain
$gf=\IdentityMap_X$.
Interchanging
$f_k$
and
$g_k$
proves
$fg=\IdentityMap_X$.
Thus
$f\in\IsometryGroup(X,d)$.
Condition~\ref{cond:rigid-basepoint} passes to the limit since
$f_k(o)\to f(o)$,
$g_k(o)\to g(o)=f^{-1}(o)$,
and
$K_n$
is closed.
For condition~\ref{cond:rigid-displacement},
uniform convergence and the reverse triangle inequality yield
\[
\begin{aligned}
&|d_k(a_j,f_k(a_j))-d(a_j,f(a_j))|\\
&\quad\le |d_k(a_j,f_k(a_j))-d(a_j,f_k(a_j))|
+|d(a_j,f_k(a_j))-d(a_j,f(a_j))|\\
&\quad\le\eta_k+d(f_k(a_j),f(a_j))\longrightarrow0.
\end{aligned}
\]
Since
$d_k(a_j,f_k(a_j))\ge2^{-m}$
for every
$k$,
we obtain
$d(a_j,f(a_j))\ge2^{-m}$.
Thus condition~\ref{cond:rigid-displacement} also holds.
Hence
$d\in \IsometryWitnessMetrics_{j,m,n}$.
Equation~\eqref{eq:rigid-complement} proves the proposition.
\end{proof}
This argument uses the compactness mechanism of \cite[Proposition~2.11]{Niemiec2014}, with additive uniform convergence in place of a common multiplicative bound.

For density, we use the following precise consequence of \cite[Theorem~1.7]{Niemiec2014}.
If
$X$
is locally compact Polish with
$|X|>2$,
then for every
$d\in\ProperMetrics(X)$
and every
$t>0$
there exists
$e\in\RigidMetrics(X)\cap\ProperMetrics(X)$
such that
$d\le e\le(1+t)d$.
Apply that theorem with
$G=\{\IdentityMap_X\}$.
The group is closed in the compact-open topology,
and each family of maps in its condition \textup{(Iso2)} is empty or a singleton,
so the required continuity and compactness conditions hold.
Condition \textup{(Iso3)} asks whether every map that preserves each two-point subset belongs to
$G$.
This holds because
$|X|>2$.

\begin{theorem}\label{thm:compact-rigid}
Let
$X$
be a compact metrizable space with
$|X|>2$.
Then
$\RigidMetrics(X)$
is a dense
$G_\delta$
subset of
$\CompatibleMetrics(X)$.
\end{theorem}
\begin{proof}
Every compatible metric on
$X$
is proper, so Proposition~\ref{prop:proper-gdelta} proves the
$G_\delta$
assertion.
Fix
$d\in\CompatibleMetrics(X)$
and
$\varepsilon>0$.
Choose
$t>0$
with
$t\diam_d(X)<\varepsilon$.
By \cite[Theorem~1.7]{Niemiec2014},
there exists
$e\in\RigidMetrics(X)$
with
$d\le e\le(1+t)d$.
Thus
\[
\UniformDistance_X(d,e)\le t\diam_d(X)<\varepsilon.
\]
This completes the proof.
\end{proof}

Compact completion extends this approximation result to all totally bounded target metrics.
\begin{proposition}\label{prop:tb-rigid}
Let
$X$
be a metrizable space with
$|X|>2$.
Then, for every
$d\in\TotallyBoundedMetrics(X)$
and every
$\varepsilon>0$,
there exists
$e\in\TotallyBoundedMetrics(X)$
with
$\UniformDistance_X(d,e)<\varepsilon$
such that the compact completion of
$(X,e)$
has trivial isometry group.
In particular,
$\RigidMetrics(X)\cap\TotallyBoundedMetrics(X)$
is dense in
$\TotallyBoundedMetrics(X)$.
\end{proposition}
\begin{proof}
Let
$(K,\widehat d)$
be the compact completion of
$(X,d)$,
identifying
$X$
with its dense image.
By Theorem~\ref{thm:compact-rigid}, choose a compatible rigid metric
$\rho$
on
$K$
with
$\UniformDistance_K(\widehat d,\rho)<\varepsilon$.
Put
$e=\rho|_{X^2}$.
Then
$e$
is compatible, totally bounded, and satisfies
$\UniformDistance_X(d,e)<\varepsilon$.
Since
$\rho$
induces the topology of
$K$,
the set
$X$
is dense in
$(K,\rho)$,
which is the completion of
$(X,e)$.
Every bijective isometry of
$(X,e)$
extends to a bijective isometry of
$(K,\rho)$
by extending both the map and its inverse.
Rigidity of
$(K,\rho)$
implies rigidity of
$(X,e)$.
\end{proof}

\begin{remark}\label{rem:smooth-rigidity}
For comparison,
in the smooth setting, Mounoud proved that metrics with trivial isometry group contain an open dense subset of the space of smooth pseudo-Riemannian metrics of any fixed admissible signature on a compact manifold of dimension at least two \cite[Theorem~1]{Mounoud2015}.
The topology there is the Whitney
$C^\infty$
topology.
\end{remark}
\section{Questions}\label{sec:questions}
The compact and totally bounded approximation results leave open the case of a general target metric.
The
$G_\delta$
argument for proper metrics also does not describe the full set of rigid metrics.
We therefore ask the following questions in the uniform topology.

\begin{question}\label{q:density}
Let
$X$
be a metrizable space with
$|X|\ge3$.
Is
$\RigidMetrics(X)$
dense in
$(\CompatibleMetrics(X),\UniformDistance_X)$?
\end{question}
\begin{question}\label{q:borel}
Let
$X$
be a metrizable space.
Is
$\RigidMetrics(X)$
a Borel subset of
$(\CompatibleMetrics(X),\UniformDistance_X)$?
\end{question}
For
$|X|\le1$,
every metric is rigid.
For
$|X|=2$,
the transposition is always an isometry, so
$\RigidMetrics(X)=\emptyset$.

Theorem~\ref{thm:density} answers Question~\ref{q:density} when
$X$
is strongly zero-dimensional and
$3\le|X|\le\mathfrak c$.
Under the additional assumption of
$\sigma$-compactness, Theorem~\ref{thm:ai-gdelta} exhibits a dense
$G_\delta$
subset of
$\RigidMetrics(X)$.
Containment of such a subset does not imply that
$\RigidMetrics(X)$
itself is Borel.

For locally compact Polish spaces
$X$
with
$|X|>2$,
\cite[Corollary~1.9]{Niemiec2014} asserts density among proper metrics for uniform convergence on compact subsets of
$X^2$.
On an unbounded space, the estimate
$d\le e\le(1+t)d$
does not by itself bound
$\UniformDistance_X(d,e)$.
Moreover, proper metrics need not be dense in
$\CompatibleMetrics(X)$.
For example,
if
$X$
is an infinite discrete space,
then the discrete metric
$d_0$
has a neighborhood in
$\CompatibleMetrics(X)$
containing no proper metric.

The proof of Proposition~\ref{prop:proper-gdelta} depends on the compactness of the sets containing
$f_k(x)$
and
$f_k^{-1}(x)$
in \eqref{eq:proper-bound}.
Completeness without properness does not provide this compactness.
\printbibliography
\end{document}